\documentclass{amsart}

\usepackage{amsmath, amssymb, amsthm, enumerate}
\usepackage{graphicx, color}
\usepackage[all]{xy}

\theoremstyle{plain}

\newtheorem{thm}{Theorem}[section]
\newtheorem{lem}[thm]{Lemma}
\newtheorem{prop}[thm]{Proposition}

\newtheorem{cor}[thm]{Corollary}

\newtheorem*{conj}{Conjecture}

\theoremstyle{definition}

\newtheorem{Def}[thm]{Definition}
\newtheorem{rem}[thm]{Remark}

\newcommand{\R}{\mathbb{R}}

\newcommand{\Z}{\mathbb{Z}}

\newcommand{\calB}{\mathcal{B}}
\newcommand{\calC}{\mathcal{C}}
\newcommand{\calO}{\mathcal{O}}

\newcommand{\abs}[1]{\lvert {#1} \rvert}
\newcommand{\EC}[2]{\mathrm{EC}(#2 ,B^{#1})}
\newcommand{\bbemb}[2]{\mathrm{Emb}_*(S^{#2},S^{#1})}
\newcommand{\bbembd}[2]{\mathrm{Emb}'_*(S^{#2},S^{#1})}
\newcommand{\bbfemb}[2]{\widetilde{\mathrm{Emb}}_*(S^{#2},S^{#1})}
\newcommand{\bbfembd}[2]{\widetilde{\rm Emb}{}'_*(S^{#2},S^{#1})}
\newcommand{\cl}{\mathrm{cl}}
\newcommand{\Emb}{\mathrm{Emb}}
\newcommand{\GL}{\mathit{GL}}
\newcommand{\GS}{\mathit{GS}}
\newcommand{\id}{\mathrm{id}}

\newcommand{\Int}{\mathrm{Int}\,}
\newcommand{\K}{\mathcal{K}}
\newcommand{\SO}{\mathit{SO}}

\numberwithin{equation}{section}
\numberwithin{figure}{section}
\numberwithin{table}{section}

\title{BV-structures on the homology of the framed long knot space}
\author{Keiichi Sakai}

\address{Department of Mathematical Sciences, Shinshu University, 3-1-1 Asahi, Matsumoto, Nagano 390-8621, Japan}
\email{ksakai@math.shinshu-u.ac.jp}
\urladdr{http://math.shinshu-u.ac.jp/~ksakai/index.html}


\begin{document}
\maketitle

\begin{abstract}
We introduce BV-algebra structures on the homology of the space of framed long knots in $\R^n$ in two ways.
The first one is given in a similar fashion to Chas-Sullivan's string topology. 
The second one is defined on the Hochschild homology associated with a cyclic, multiplicative operad of graded modules.
The latter can be applied to Bousfield-Salvatore spectral sequence converging to the homology of the space of framed long knots.
Conjecturally these two structures coincide with each other.
\end{abstract}

\section{Introduction}\label{s:intro}
The space of framed long embeddings is known to be acted on by the little disks operad \cite{Budney03}.
A natural question is whether this action extends to any action of the {\em framed} little disks operad.
The answer seems affirmative, in view of \cite{Budney08, Salvatore06, Salvatore09}, \cite{DwyerHess10,SalvatoreWahl03}.

In the first result of this paper (Theorem~\ref{thm:BV}) we imitate Chas-Sullivan's string topology \cite{ChasSullivan99} to realize Salvatore's homotopy-theoretical action in a geometric and homological way.
Namely we define a {\em BV-algebra} structure on the homology of the space of framed long knots.
Our BV-structure is outlined as follows.
The bracket (called {\em Poisson bracket}) is induced by an action of little $2$-disks operad \cite{Budney03}. 
The {\em BV-operation} (usually denoted by $\Delta$) is derived from Hatcher's cycle \cite[p.~3]{Hatcher02}, which in a sense ``pushes the base point through long knots''.
As a corollary we obtain a Lie algebra structure on the $S^1$-equivariant homology.
We also show that our BV-operation $\Delta$ is not trivial (Proposition~\ref{prop:Delta_nonzero} below).

Our second result (Theorem~\ref{thm:BV_Hochschild}) is an algebraic one.
Based on \cite{Salvatore09,Bousfield87}, a homology spectral sequence, converging to the homology of the space of framed long knots (at least in higher-codimension cases), is constructed.
Its $E^2$-term is the {\em Hochschild homology} $HH_*(H_*(f\calC ))$ associated with the homology operad $H_*(f\calC )$ of the framed little disks operad, which is cyclic \cite{Budney08} and multiplicative.
Motivated by these facts, in \S\ref{s:BV_Hochschild} we provide a BV-algebra structure on $HH_*(\calO )$ of any cyclic and multiplicative operad $\calO$ of {\em graded} modules.
A bracket has already been defined in \cite{Tourtchine05}, and our BV-operation is given by a graded version of Connes' boundary operator (see \cite{Loday92}).
Our proof is a direct analogue to that for non-graded cases \cite{Menichi04}. 
Presumably Salvatore's framed little disks action would deduce the same formula as ours.

The paper is organized as follows.
In \S\ref{s:notations} three spaces of framed embeddings are defined and proved to be homotopy equivalent to each other.
We describe our geometric BV-algebra structure explicitly in \S\ref{s:BV}; the Poisson bracket in \S\ref{ss:Poisson} and \S\ref{ss:Poisson_emb}, and the $\Delta$-operation in \S\ref{ss:Hatcher}.
The definition of $HH_*(\calO)$ is reviewed in \S\ref{ss:Hochschild}, and the BV-algebra structure on $HH_*(\calO)$ is defined in \S\ref{ss:Connes}.

\section{The spaces of framed long knots}\label{s:notations}
We denote
by $B^n:=\{ x\in\R^n\mid\abs{x}\le 1\}$ the $n$-ball, and $S^n:=\partial B^{n+1}\subset\R^{n+1}$.
We often write $\infty :=(0,\dotsc,0,1)\in S^n$.
$S^1$ is always identified with $\R/\Z$ and let $0=1\in S^1$ serve as the basepoint.

We define three spaces of framed long knots, $\bbfemb{n}{1}$, $\bbfembd{n}{1}$ and $\EC{n-1}{1}$.
Eventually they turn out to be homotopy equivalent to each other.
A convenient one will be used to construct each homology operation.

First we define $\EC{n-1}{1}$, originally introduced in \cite{Budney03}.
The homotopy type of this space can be nicely described through an action of the little disks operad, and hence its homology is equipped with a Poisson algebra structure (see \S\ref{ss:Poisson}).
\begin{Def}[\cite{Budney03}]\label{def:EC}
For a manifold $M$, define the space ${\rm EC}(k,M)$ by
\[
 {\rm EC}(k,M):=\{ f:\R^k\times M\hookrightarrow\R^k\times M\, |\,
 f(t;x)=(t;x)\text{ if }t\not\in[0,1]^k\}.
\]
\end{Def}

We consider the case of {\em framed long knots}, namely $M=B^{n-1}$ and $k=1$.
The space $\EC{n-1}{1}$ is related to the space of {\em long knots}
\[
 \K_n:=\{ f:\R^1\hookrightarrow \R^1\times B^{n-1}\mid f(t)=(t,0,\dotsc,0)\text{ if }t\not\in[0,1]\}
\]
via the restriction map
\[
 {\rm res}:\EC{n-1}{1}\to\K_n,\quad f\mapsto f|_{\R^1\times\{ 0\}}
\]
which is a fibration with fiber $\Omega\SO(n-1)$.

\begin{Def}\label{def:Embd}
Define $\iota:S^1\hookrightarrow S^n$ and $\tilde{\iota}:S^1\times B^{n-1}\hookrightarrow S^n$ by
\begin{align*}
 \iota (t)&:=(\sin 2\pi t,0,\dots ,0,\cos 2\pi t)\in S^n\subset\R^{n+1},\quad\text{and}\\
 \tilde{\iota}(t;x)&:=\frac{\iota(t)+\epsilon_0(0,x,0)}{\abs{\iota(t)+\epsilon_0(0,x,0)}},
\end{align*}
where $\epsilon_0>0$ is some fixed small number.
Define $\bbfembd{n}{1}$ to be the space of embeddings $\tilde{\varphi}:S^1\times B^{n-1}\hookrightarrow S^n$ such that $\tilde{\varphi}(0;x)=\tilde{\iota}(0;x)$ and all the partial derivatives of $\tilde{\varphi}$ at $(0;x)\in S^1\times B^{n-1}$ of all orders are equal to those of $\tilde{\iota}$.
Similarly define $\bbembd{n}{1}$ to be the space of all embeddings $\varphi:S^1\hookrightarrow S^n$ such that $\varphi(0)=\iota(0)=\infty$ and all of its derivatives at $t=0$ are equal to those of $\iota$.
\end{Def}

The restriction map
\[
 \mathrm{res}:\bbfembd{n}{1}\to\bbembd{n}{1},\quad
 \tilde{\varphi}\mapsto\tilde{\varphi}|_{S^1\times\{0\}},
\]
is also a fibration with fiber $\Omega\SO(n-1)$.

There is another embedding space on which $S^1$ acts (Lemma~\ref{lem:S1-action}):

\begin{Def}\label{def:Emb}
Define $\bbemb{n}{1}$ to be the space of embeddings $\varphi:S^1\hookrightarrow S^n$ satisfying
\[
 \varphi (0)=\infty ,\quad
 \varphi'(0)/\abs{\varphi'(0)}=(1,0,\dotsc,0).
\]
Define $\bbfemb{n}{1}$ to be the space of pairs $(\varphi ;w)$, where $\varphi\in\bbemb{n}{1}$ and $w=(w_0,\dotsc,w_n):S^1\to\SO(n+1)$ ($w_i:S^1\to S^n$) satisfying
\begin{itemize}
\item $w_0(t)=\varphi'(t)/\abs{\varphi'(t)}$ and $w_n(t)=\varphi(t)$ for any $t\in S^1$,
\item $w(0)=I_{n+1}$ (the identity matrix).
\end{itemize}
Denote by $\pi:\bbfemb{n}{1}\to\bbemb{n}{1}$ the natural projection.
\end{Def}

Define $\tilde{\cl}:\EC{n-1}{1}\to\bbfembd{n}{1}$ (``closure'' of long embeddings) by $\tilde{\cl}(f):=\tilde{\iota}\circ f$.
This is smooth at $t=0$ and has the same partial derivatives at $(0;x)\in S^1\times B^{n-1}$ as $\tilde{\iota}$ because $f\in\EC{n-1}{1}$ extends to $\id_{\R^1\times B^{n-1}}$ outside $[0,1]$.
Similarly define the map $\cl:\K_n\to\bbembd{n}{1}$ by $\cl(f):=\tilde{\iota}\circ f$.

Define $\tilde{h}:\bbfembd{n}{1}\to\bbfemb{n}{1}$ by
\[
 \tilde{h}(\tilde{\varphi})(t):=\left(\tilde{\varphi}(t;0);\GS\left(\frac{\partial\tilde{\varphi}}{\partial t}(t;0),\frac{\partial\tilde{\varphi}}{\partial x_1}(t;0),\dotsc,\frac{\partial\tilde{\varphi}}{\partial x_{n-1}}(t;0),\tilde{\varphi}(t;0)\right)\right),
\]
where $\GS:\GL^+_{n+1}(\R )\xrightarrow{\simeq}\SO(n+1)$ is the Gram-Schmidt orthonormalization,
and define $h:\bbembd{n}{1}\to\bbemb{n}{1}$ as the natural inclusion.
Note that we have maps of fibration sequences;
\begin{equation}\label{eq:fibration_diagram2}
 \begin{split}
 \xymatrix{
 \Omega\SO(n-1)\ar@{=}[d]\ar[r] & \EC{n-1}{1}\ar[d]_-{\tilde{\cl}}\ar[r]^-{\rm res} & \K_n\ar[d]_-{\cl}\\
 \Omega\SO(n-1)\ar@{=}[d]\ar[r] & \bbfembd{n}{1}\ar[r]^-{\rm res}\ar[d]_-{\widetilde{h}} & \Emb'_*(S^1,S^n)\ar[d]_-h \\
 \Omega\SO(n-1)\ar[r]           & \bbfemb{n}{1}\ar[r]^-{\pi} & \bbemb{n}{1}
 }
 \end{split}
\end{equation}

\begin{prop}\label{prop:he}
The maps $\tilde{\cl}$ and $\cl$ are homotopy equivalences.
\end{prop}
\begin{proof}
Because the embedding spaces have homotopy types of CW-complexes, it suffices to show that the maps $\tilde{\cl}$ and $\cl$ are weak homotopy equivalences.
By \eqref{eq:fibration_diagram2}, it is enough to show that $\cl$ is a (weak) homotopy equivalence.

The homotopy inverse $\cl^{-1}$ is given by
\[
 \cl(\varphi)(t):=
 \begin{cases}
  \Phi(\varphi(t)) & \text{if }t\in[0,1],\\
  (t,0,\dotsc,0)   & \text{if }t\not\in[0,1],
 \end{cases}
\]
where $\Phi:S^n\setminus\{\infty\}\xrightarrow{\cong}(0,1)\times\Int B^{n-1}$ is the stereographic projection $S^n\setminus\{\infty\}\xrightarrow{\cong}\R^n$ followed by a dilation $\R^n\xrightarrow{\cong}(0,1)\times\Int B^{n-1}$ which is chosen so that $\Phi(\iota(t))=(t,0,\dotsc,0)$ for $t\in[0,1]$.
\end{proof}

\begin{prop}\label{prop:all_he}
The maps $\tilde{h}$ and $h$ are homotopy equivalences.
\end{prop}
\begin{proof}
Similarly to the proof of Proposition~\ref{prop:he}, it is enough to show that $h$ is a weak homotopy equivalence.
The idea is to ``straighten'' the elements of $\bbemb{n}{1}$ around $t=0$.

First we show that $h_*:\pi_k(\bbembd{n}{1},\iota)\to\pi_k(\bbemb{n}{1},\iota)$ is surjective (the basepoint $\iota$ will be omitted below).
Let $\xi\in\pi_k(\bbemb{n}{1})$ be represented by $g:S^k\to\bbemb{n}{1}$.
For each $z\in S^k$, there exists $\epsilon>0$ such that the $\epsilon$-ball $B_{\epsilon}\subset S^n$ (with respect to the standard metric on $S^n$) centered at $\infty\in S^n$ satisfies that
\begin{itemize}
\item $g(z)(S^1)\cap B_{\epsilon}$ is connected, and
\item if $g(z)^{-1}(B_{\epsilon})=(-\delta_1,\delta_2)$, then $g(z)_1$ (the first coordinate of $g(z)$) monotonely increases on the interval $(-\delta_1,\delta_2)$.
\end{itemize}
Such an $\epsilon>0$ as above can be taken uniformly for all $z\in S^k$.
Indeed there is $\epsilon'>0$ such that $g(z)'_1(t)>0$ for all $(z,t)\in S^k\times(-\epsilon',\epsilon')$, because the map $S^k\times S^1\to\R$ given by $(z,t)\mapsto g(z)'_1(t)$ is continuous and $\{g(z)'_1(0)\mid z\in S^k\}$ has the positive minimum by the compactness of $S^k$.
Then we can take $\epsilon>0$ so that $B_{\epsilon}$ does not intersect the compact set $\hat{g}(S^k\times(S^1\setminus(-\epsilon',\epsilon'))$, where $\hat{g}(z,t):=g(z)(t)$.
This $\epsilon$ satisfies the above conditions.
Note that $\delta_1,\delta_2>0$ depend continuously on $z$.

Consider a natural projection $p:B_{\epsilon}\to\iota (S^1)\cap B_{\epsilon}$ and diffeomorphisms $s_z:(-\delta_1,\delta_2)\to(-\delta'_1,\delta'_2)$ for some $\delta'_1,\delta'_2>0$ such that $p(g(z)(t))=\iota(s_z(t))$ (see Figure~\ref{fig:nbd}).
\begin{figure}
\begin{center}
\includegraphics{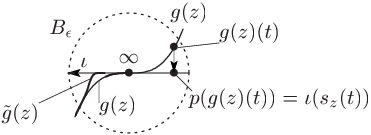}
\end{center}
\caption{A modification of $g$ to be standard near $\infty$}
\label{fig:nbd}
\end{figure}
Putting $\delta :=\min_{z\in S^k}\{\delta_1,\delta_2\}$, define
\[
 \bar{g}(z)(t):=
 \frac{(1-b_{\delta}(t))g(z)(t)+b_{\delta}(t)\iota(s_z(t))}{\abs{(1-b_{\delta}(t))g(z)(t)+b_{\delta}(t)\iota(s_z(t))}}
\]
where $b_{\delta}(t):=b(t/\delta)$ and $b:\R\to\R_{\ge 0}$ is a fixed bump function with support $\abs{t}\le 1$ satisfying $b(t)\equiv 1$ for $\abs{t}\le 1/2$ (and hence $b_{\delta}(t)\equiv 1$ for $\abs{t}\le\delta/2$).
By construction $\bar{g}$ is homotopic to $g$ and $\bar{g}(z)(t)=\iota(s_z(t))$ on $(-\delta,\delta)$.
Define
\[
 \tilde{g}(z)(t):=\bar{g}(z)(b_{\delta/2}(t)t+(1-b_{\delta/2}(t))s_z(t))
\]
then $g\sim\tilde{g}$ and $\tilde{g}(z)=\iota$ on $(-\delta/4,\delta/4)$, and hence $\tilde{g}$ represents an element of $\pi_k(\bbembd{n}{1})$ which is mapped to $\xi=[g]$.
Thus surjectivity follows.

Injectivity is proved in a similar way.
Suppose $\eta=[g]\in\pi_k(\bbembd{n}{1})$ maps to $0\in\pi_k(\bbemb{n}{1})$, and choose a map $G:B^{k+1}\to\bbemb{n}{1}$ bounded by $g$.
Since $B^{k+1}$ is compact, we can deform $G$ to be standard near $t=0$ in a similar way to the above.
Thus $g$ bounds a ball in $\bbembd{n}{1}$ and hence $\eta=0$.
\end{proof}

\begin{rem}
In fact the homotopy inverse $\tilde{\cl}{}^{-1}$ is given by composing an appropriate diffeomorphism $\Psi:S^n\setminus\tilde{\iota}(\{0\}\times B^{n-1})\xrightarrow{\cong}(0,1)\times\Int B^{n-1}$ to the elements of $\bbfembd{n}{1}$.
The homotopy inverse $\tilde{h}{}^{-1}$ is given by ``straightening'' the elements of $\bbfemb{n}{1}$ and ``fattening'' the knots by geodesics along the framings.
\end{rem}

\section{A geometric BV-structure}\label{s:BV}

\begin{Def}\label{def:BV_algebra}
A {\em $k$-Poisson (-Gerstenhaber) algebra} $A$ is a graded commutative algebra equipped with a graded Lie bracket $[-,-]:A\times A\to A$ of degree $k$, called \emph{Poisson (-Gerstenhaber) bracket}, satisfying the Leibniz rule
\[
 [x,yz]=[x,y]z+(-1)^{(\tilde{x}+k)\tilde{y}}y[x,z],
\]
where $\tilde{x}$ denotes the degree of $x$, that is, $x\in A_{\tilde{x}}$.
A $1$-Poisson algebra $A$ is called a \emph{BV-algebra} if it is endowed with a degree one operation $\Delta :A\to A$ satisfying
\begin{itemize}
\item $\Delta\circ\Delta =0$,
\item $\Delta (xy)=\Delta (x)y+(-1)^{\tilde{x}}x\Delta (y)+(-1)^{\tilde{x}}[x,y]$.
\end{itemize}
\end{Def}

The aim of this section is to define a BV-algebra structure on $H_*(\bbfemb{n}{1})$.
A Poisson algebra structure on $H_*(\EC{n-1}{1})$ has already been defined in \cite{Budney08}.
First we describe this structure on $H_*(\bbfembd{n}{1})$ (\S\ref{ss:Poisson_emb}).
Then we define the $\Delta$-operation using an $S^1$-action on $\bbfemb{n}{1}$.

\subsection{Budney's Poisson structure}\label{ss:Poisson}
Budney \cite{Budney03} constructed an action of the little $(k+1)$-disks operad $\calC_{k+1}$ on ${\rm EC}(k,M)$.
The main idea of the action can be found in \cite[Fig.~2]{Budney03};
start with the {\em connected-sum} $f\sharp g$ (defined explicitly below; see Figure~\ref{fig:connected_sum}), ``push off'' $g$ through $f$ as in the right-half of \cite[Fig.~2]{Budney03} until we arrive at $g\sharp f$, and perform the same procedure with $f$ and $g$ exchanged.
As a corollary we have the following.

\begin{thm}[\cite{Budney03}]\label{thm:Budney}
$H_*({\rm EC}(k,M))$ admits a $k$-Poisson algebra structure.
\end{thm}

Here we describe Budney's Poisson algebra structure on $H_*(\EC{n-1}{1})$ explicitly.
The product (denoted by $x\cdot y$, or simply $xy$) and the Poisson bracket (denoted by $\lambda$) are induced by the ``second stage'' of the action of $\calC_2$;
\[
 \mu_2:\calC_2(2)\times\EC{n-1}{1}^{\times 2}\to\EC{n-1}{1}.
\]
The product corresponds to the generator of $H_0(\calC_2(2))\cong\Z$ and is induced by the {\em connected-sum} defined as follows.
For $0\le\alpha\le 1$, define two diffeomorphisms
\begin{alignat*}{2}
 &l_{\alpha}:\R^1\xrightarrow{\cong}\R^1&\quad\text{by}\quad &l_{\alpha}(t):=2t-\alpha,\\
 &s_{\alpha}:\R^1\times B^{n-1}\xrightarrow{\cong}\R^1\times B^{n-1}&\quad\text{by}\quad &s_{\alpha}:=l_{\alpha}\times\id_{B^{n-1}}.
\end{alignat*}
Then for $f\in\EC{n-1}{1}$,
\[
 L_{\alpha}f:=s_{\alpha}^{-1}\circ f\circ s_{\alpha}\in\EC{n-1}{1}
\]
has the support $[\frac{\alpha}{2},\frac{\alpha+1}{2}]\times B^{n-1}$ and satisfies
\[
 L_{\alpha}f\bigl(\bigl[\frac{\alpha}{2},\frac{\alpha+1}{2}\bigr]\times B^{n-1}\bigr)\subset\bigl[\frac{\alpha}{2},\frac{\alpha+1}{2}\bigr]\times B^{n-1}
\]
(see \cite[Fig.~4]{Budney03}).
Then for $f,g\in\EC{n-1}{1}$, define $f\sharp g\in\EC{n-1}{1}$ by
\[
 f\sharp g(t;x):=
 \begin{cases}
  L_0f(t;x) & t\le\frac{1}{2},\\
  L_1g(t;x) & t\ge\frac{1}{2}
 \end{cases}
\]
(see Figure~\ref{fig:connected_sum}).
\begin{figure}
\includegraphics[scale=0.7]{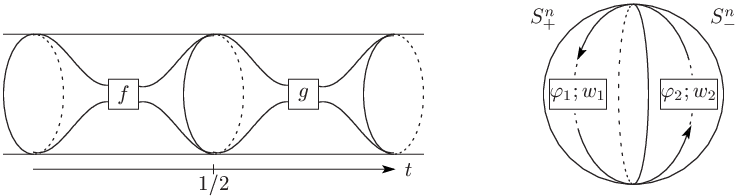}
\caption{Connected-sums on ${\rm EC}(k,M)$ and $\bbfemb{n}{1}$}
\label{fig:connected_sum}
\end{figure}
We notice that the elements in $\EC{n-1}{1}$ are maps $\R^1\times B^{n-1}\to\R^1\times B^{n-1}$ and we can compose them.
Using the composition, we can write $f\sharp g$ as
\[
 f\sharp g =(L_0f)\circ(L_1g) =(L_1g)\circ(L_0f).
\]

Poisson bracket $\lambda$ corresponds to the generator of $H_1(\calC_2(2))\cong\Z$ and can be described as follows.
For $0\le\alpha\le 1$ and $f,g\in\EC{n-1}{1}$, define $f*_{\alpha}g\in\EC{n-1}{1}$ by
\[
 f*_{\alpha}g:=(L_{\alpha}f)\circ(L_{1-\alpha}g).
\]
This defines the {\em $*$-operation}
\[
 *:I\times\EC{n-1}{1}^{\times 2}\to\EC{n-1}{1},\quad
 (\alpha ,f,g)\mapsto f*_{\alpha}g.
\]

\begin{rem}\label{rem:commutative}
The $*$-operation gives a homotopy between $f*_0g=f\sharp g$ and $f*_1g=g\sharp f$, and hence $\sharp$ is homotopy commutative.
In particular for homology classes $x,y$, we have $x*_{\alpha}y=xy$ for any $0\le\alpha\le 1$.
\end{rem}

The map $S^1\times\EC{n-1}{1}^{\times 2}\to\EC{n-1}{1}$ defined by
\[
 (\alpha,f,g)\mapsto
 \begin{cases}
  f*_{2\alpha}g   & 0\le\alpha\le\frac{1}{2},\\
  g*_{2\alpha-1}f & \frac{1}{2}\le\alpha\le 1
 \end{cases}
\]
corresponds to the map $\mu_2$ appearing in Budney's $\calC_2$-action, and the Poisson bracket $\lambda$ is induced by this map and a fixed generator of $H_1(S^1)\cong\Z$.

For any cubical chains $\xi :I^p\to\EC{n-1}{1}$ and $\eta :I^q\to\EC{n-1}{1}$, define cubical $(p+q+1)$-chains $\xi *\eta$ and $\eta*'\xi$ by
\begin{align*}
 &\xi*\eta:I\times I^p\times I^q\ni(\alpha,u,v)\mapsto\xi(u)*_{\alpha}\eta(v),\\
 &\eta*'\xi:I\times I^p\times I^q\ni(\alpha,u,v)\mapsto\eta(v)*_{\alpha}\xi(u),
\end{align*}
and extend them linearly on the cubical chain complex.
We also define cubical $(p+q)$-chains $\xi*_{\alpha}\eta:=\xi*\eta|_{\{\alpha\}\times I^{p+q}}$ and $\eta*'_{\alpha}\xi:=\eta*'\xi|_{\{\alpha\}\times I^{p+q}}$.
If $x$ and $y$ are cubical $p$- and $q$-cycles, then $\partial (x*y)=x*_1y-x*_0y$ and $\partial(y*'x)=y*'_1x-y*'_0x$.
Since $y*'_1x=x*_0y$ and $y*'_0x=x*_1y$,
\begin{equation}\label{eq:bracket}
 (-1)^{p-1}(x*y+y*'x)
\end{equation}
is a cubical $(p+q+1)$-cycle.
The cycle \eqref{eq:bracket} represents $\lambda (x,y)$.

\subsection{Poisson structure for $\bbfembd{n}{1}$}\label{ss:Poisson_emb}
We can transfer the above Poisson structure on $H_*(\EC{n-1}{1})$ to $H_*(\bbfembd{n}{1})$ via the homotopy equivalence $\tilde{\cl}$ defined in Proposition~\ref{prop:he}.
We illustrate this structure from a geometric view.

Figure~\ref{fig:connected_sum} explains the idea of connected-sum on $\bbfembd{n}{1}$.
For $\tilde{\varphi}\in\bbfembd{n}{1}$ and $0\le\alpha\le 1$, define
\[
 L_{\alpha}\tilde{\varphi}:=\tilde{\cl}(L_{\alpha}\tilde{\cl}{}^{-1}(\tilde{\varphi})).
\]
Roughly speaking $L_{\alpha}\tilde{\varphi}$ is $\tilde{\iota}$ with the cylinder $\tilde{\iota}([\frac{\alpha}{2},\frac{\alpha+1}{2}]\times B^{n-1})\cong I\times B^{n-1}$ replaced by $\tilde{\cl}{}^{-1}(\tilde{\varphi})$.
Then for $\tilde{\varphi}_1,\tilde{\varphi}_2\in\bbfembd{n}{1}$,
\begin{equation}\label{eq:connected-sum}
 \tilde{\varphi}_1\sharp\tilde{\varphi}_2:=(L_0\tilde{\varphi}_1)\circ(L_1\tilde{\varphi}_2)=
 \begin{cases}
  L_0\tilde{\varphi}_1 & \text{on }[0,\frac{1}{2}]\times B^{n-1},\\
  L_1\tilde{\varphi}_2 & \text{on }[\frac{1}{2},1]\times B^{n-1}.
 \end{cases}
\end{equation}

The $*$-operation for $\bbfembd{n}{1}$
\[
 *:I\times \bbfembd{n}{1}^{\times 2}\to\bbfembd{n}{1}
\]
is defined by using that for $\EC{n-1}{1}$;
\begin{equation}\label{eq:*-operation}
\begin{split}
 (\alpha ,\tilde{\varphi}_1,\tilde{\varphi}_2)\mapsto
 \tilde{\varphi}_1*_{\alpha}\tilde{\varphi}_2&:=\tilde{\cl}(\tilde{\cl}{}^{-1}(\tilde{\varphi}_1)*_{\alpha}\tilde{\cl}{}^{-1}(\tilde{\varphi}_2))\\
 &=(L_{\alpha}\tilde{\varphi}_1)\circ(L_{1-\alpha}\tilde{\cl}{}^{-1}(\tilde{\varphi}_2)).
\end{split}
\end{equation}
Roughly speaking $\tilde{\varphi}_1*_{\alpha}\tilde{\varphi}_2$ is $L_{\alpha}\tilde{\varphi}_1$ with $L_{\alpha}\tilde{\varphi}_1([\frac{1-\alpha}{2},\frac{2-\alpha}{2}]\times B^{n-1})$ replaced by $\cl{}^{-1}(\tilde{\varphi}_2)$.

\subsection{BV-operation}\label{ss:Hatcher}
We define an $S^1$-action on $\bbfemb{n}{1}$ as was done in \cite[p.~3]{Hatcher02}.
This action induces our BV-operation on $H_*(\bbfemb{n}{1})$.

For any $(\varphi ;w)\in\bbfemb{n}{1}$ and $\alpha\in S^1$, define $(\varphi ;w)^{\alpha}\in\bbfemb{n}{1}$ by
\[
 (\varphi;w)^{\alpha}(t):=(A\varphi(t-\alpha);Aw(t-\alpha)),
\]
where $A=w(-\alpha)^{-1}\in\SO(n+1)$ (acting on $S^n$ in the usual way).
Since $Aw(-\alpha)=I_{n+1}$ (and hence $A\varphi(-\alpha )=\infty$), $(\varphi;w)^{\alpha}$ is indeed in $\bbfemb{n}{1}$.

\begin{lem}\label{lem:S1-action}
The above formula defines an $S^1$-action on $\bbfemb{n}{1}$.
That is, we have $((\varphi ;w)^{\alpha})^{\beta}=(\varphi ;w)^{\alpha +\beta}$ and $(\varphi ;w)^0=(\varphi ;w)$.
\end{lem}

This action can be interpreted for $\tilde{\varphi}\in\bbfembd{n}{1}$ via $\tilde{h}$ (see \eqref{eq:fibration_diagram2});
\[
 \tilde{\varphi}{}^{\alpha}:=\tilde{h}{}^{-1}(\tilde{h}(\tilde{\varphi})^{\alpha}).
\]

The fundamental class of $S^1$ induces our $\Delta$-operation through the above action;
\[
 \Delta :H_*(\bbfemb{n}{1})\to H_{*+1}(\bbfemb{n}{1}).
\]
We have $\Delta^2=0$ since $\Delta$ is induced by an $S^1$-action and $H_*(S^1)=\bigwedge^*\langle [S^1]\rangle$.

\begin{thm}\label{thm:BV}
$(H_*(\bbfemb{n}{1}),\cdot ,\lambda ,\Delta )$ is a BV-algebra.
\end{thm}

\begin{proof}
We need to prove the last equality of Definition~\ref{def:BV_algebra}, that is, $\Delta$ is a derivation with respect to the product modulo $\lambda$;
\begin{equation}\label{eq:BV}
 \Delta (xy)-\Delta (x)y-(-1)^{\tilde{x}}x\Delta (y)=(-1)^{\tilde{x}}\lambda (x,y).
\end{equation}
This is proved in a similar way to \cite[Lemma~5.2]{ChasSullivan99}.
Define two operations $\Delta_i$ ($i=1,2$) as the ``first/last half'' of $\Delta$;
\[
 \Delta_i:I\times \bbfemb{n}{1}\to\bbfemb{n}{1},\quad \Delta_i(\alpha,\sigma)=\sigma^{(\alpha +i-1)/2}.
\]
Let $\Delta^2:=\{ 0\le\beta\le\alpha\le 1\}$ be the standard $2$-simplex.
Define
\[
 \Phi:\Delta^2\times \bbfemb{n}{1}^{\times 2}\to\bbfemb{n}{1}
\]
by
\[
 \Phi((\alpha ,\beta),\sigma_1,\sigma_2):=(\sigma_1*_{\alpha}\sigma_2)^{\beta /2}.
\]
\begin{figure}
\includegraphics[scale=0.6]{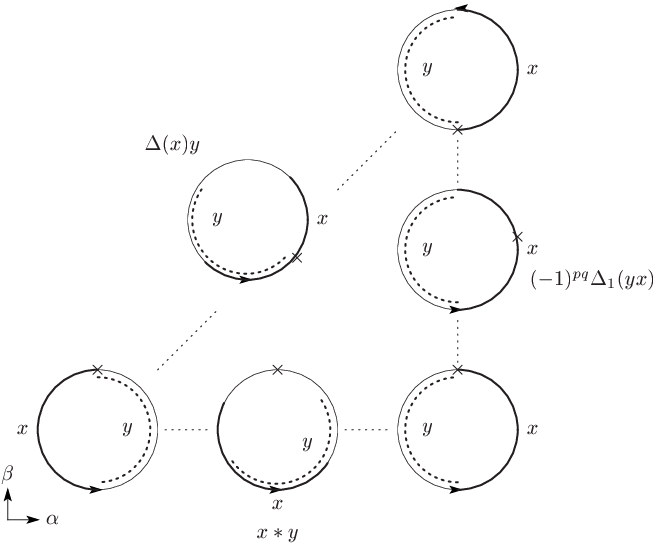}
\caption{The chain $\Phi_*(\id,x,y)$; the symbol $\times$ indicates the basepoint}
\label{fig:BV_proof}
\end{figure}Choose cubical $p$- and $q$-cycles $x=\sum a_ix_i$ and $y=\sum b_jy_j$.
Then $\Phi_*(\id,x,y)=\sum a_ib_j\Phi\circ(\id_{\Delta^2}\times x_i\times y_j)$ is a $(p+q+2)$-chain whose boundary comes from $\partial\Delta^2$.
We can see that
\begin{enumerate}[(i)]
\item $\{\alpha=\beta\}\subset\partial\Delta^2$ corresponds to a $(p+q+1)$-cycle homologous to $\Delta (x)y$,
\item $\{\alpha=1\}\cup\{\beta=0\}\subset\partial\Delta^2$ corresponds to a $(p+q+1)$-cycle homologous to $x*y+\Delta_1(x*_1y)$ (see Figure~\ref{fig:BV_proof}).
\end{enumerate}
(ii) is immediate from the definition, and (i) follows from Lemma~\ref{lem:a=b} below.
Thus
\begin{equation}\label{eq:proof1}
 x*y+\Delta_1(x*_1y)-\Delta (x)y\sim 
 0.
\end{equation}
Similarly, considering the $(p+q+2)$-chain $(y*'_{\alpha}x)^{\beta/2}$, we have
\begin{equation}\label{eq:proof2}
 y*'x+\Delta_1(y*'_1x)-\Delta(y)*'_0x\sim 0.
\end{equation}
By definition we have equalities of $(p+q+1)$-chains
\begin{equation}\label{eq:proof3}
\Delta_1(y*'_1x)=\Delta_2(y*'_0x)=\Delta_2(x*_1y).
\end{equation}
We also see, by exchanging $u\in I^p$ and $v\in I^q$, that the cycle $\Delta(y)*'_0x$ is homologous to $(-1)^{pq}\Delta(y)*_0x\sim(-1)^px\Delta(y)$.
Substituting this and \eqref{eq:proof3} into \eqref{eq:proof2} we have
\begin{equation}\label{eq:proof4}
 y*'x+\Delta_2(x*_1y)-(-1)^px\Delta(y)\sim 0.
\end{equation}
Then \eqref{eq:proof1}, \eqref{eq:proof4} and the facts $\Delta_1(z)+\Delta_2(z)\sim\Delta (z)$ (for any cycle $z$) and $x*_1y\sim xy$ (see Remark~\ref{rem:commutative}) imply \eqref{eq:BV}.
\end{proof}

\begin{lem}\label{lem:a=b}
$\Phi|_{\{\alpha=\beta\}\times\bbfemb{n}{1}^{\times 2}}:I\times\bbfembd{n}{1}^{\times 2}\to\bbfembd{n}{1}$, defined by $(\alpha,\tilde{\varphi}_1,\tilde{\varphi}_2)\mapsto(\tilde{\varphi}_1*_{\alpha}\tilde{\varphi}_2)^{\alpha/2}$, is homotopic to $(\alpha,\tilde{\varphi}_1,\tilde{\varphi}_2)\mapsto(\tilde{\varphi}_1)^{\alpha}\sharp\tilde{\varphi}_2$.
\end{lem}

\begin{proof}
For $\tilde{\varphi}_1,\tilde{\varphi}_2\in\bbfembd{n}{1}$, the embedding $\tilde{\varphi}_1*_{\alpha}\tilde{\varphi}_2$ is given by \eqref{eq:*-operation}, which is equal to $L_{\alpha}\tilde{\varphi}_1$ at $t=-\alpha/2\equiv\frac{2-\alpha}{2}$.
The embedding $\tilde{h}(\tilde{\varphi}_1*_{\alpha}\tilde{\varphi}_2)^{\alpha/2}\in\bbfemb{n}{1}$ is $\tilde{h}(\tilde{\varphi}_1*_{\alpha}\tilde{\varphi}_2)$ rotated by an action of a matrix which is determined by the value $\tilde{h}(\tilde{\varphi}_1*_{\alpha}\tilde{\varphi}_2)(-\alpha/2)=\tilde{h}(L_{\alpha}\tilde{\varphi}_1)(-\alpha/2)$.
Thus $(\tilde{\varphi}_1*_{\alpha}\tilde{\varphi}_2)^{\alpha/2}$ is homotopic to
\begin{align*}
 (t;x)\mapsto
 &A(L_{\alpha}\tilde{\varphi}_1(-\alpha/2))\cdot(L_{\alpha}\tilde{\varphi}_1(L_{1-\alpha}\tilde{\cl}{}^{-1}(\tilde{\varphi}_2)(t-\frac{\alpha}{2},x)))\\
 &=
  (L_{\alpha}\tilde{\varphi}_1)^{\alpha/2}
  (L_1\tilde{\cl}{}^{-1}(\tilde{\varphi}_2)(t,x))\\
 &\sim(L_0\tilde{\varphi}_1)^{\alpha}(L_1\tilde{\cl}{}^{-1}(\tilde{\varphi}_2)(t,x)),
\end{align*}
here the equality follows from $l_{1-\alpha}(t-\frac{\alpha}{2})=l_1(t)$ and $l_{1-\alpha}^{-1}(u)=l_1^{-1}(u)-\frac{\alpha}{2}$, and then we use Lemma~\ref{lem:shift} below.
The proof is completed by considering the homotopy from any $f\in\EC{n-1}{1}$ to $L_0f$ given by
\[
 \beta\mapsto\hat{s}_{\beta}^{-1}\circ f\circ\hat{s}_{\beta}
\]
where $\hat{s}_{\beta}:=\hat{l}_{\beta}\times\id_{B^{n-1}}$ and $\hat{l}_{\beta}(t):=(1+\beta)t$, and translating this homotopy to a homotopy on $\bbfemb{n}{1}$ via $\tilde{\cl}$.
\end{proof}

\begin{lem}\label{lem:shift}
The map $I\times\bbfemb{n}{1}\to\bbfemb{n}{1}$ given by $(\alpha,\sigma)\mapsto(L_{\alpha}\sigma)^{\alpha/2}$ (where $L_{\alpha}\sigma$ is defined through the homotopy equivalence $\tilde{h}$) is equal to $(\alpha,\sigma)\mapsto(L_0\sigma)^{\alpha}$.
\end{lem}

\begin{proof}
First by definition $L_{\alpha}\sigma(t)=R_{\alpha}\cdot L_0\sigma(t-\frac{\alpha}{2})$ where $R_{\alpha}$ is the rotation by $\pi\alpha$ in the $x_1x_{n+1}$-plane.
Putting $L_0\sigma=(\varphi;w)$, we have
\[
 (L_{\alpha}\sigma)^{\alpha/2}(t)=(AR_{\alpha}\varphi((t-\frac{\alpha}{2})-\frac{\alpha}{2});AR_{\alpha}w((t-\frac{\alpha}{2})-\frac{\alpha}{2}))
\]
where $A=(R_{\alpha}w(-\frac{\alpha}{2}-\frac{\alpha}{2}))^{-1}=w(-\alpha)^{-1}R^{-1}_{\alpha}$.
Thus
\[
 (L_{\alpha}\sigma)^{\alpha/2}(t)=(w(-\alpha)^{-1}\varphi(t-\alpha);w(-\alpha)^{-1}w(t-\alpha))=(L_0\sigma)^{\alpha}(t).\qedhere
\]
\end{proof}

\begin{prop}\label{prop:Delta_nonzero}
$\Delta$ is nontrivial when $n\ge 3$ is odd.
\end{prop}

\begin{proof}
It is an easy consequence of the third equation from Definition~\ref{def:BV_algebra} that, if $\lambda (x,y)\ne 0$, then one of $\Delta (xy)$, $\Delta (x)$ and $\Delta (y)$ is not zero.
The non-triviality of $\lambda$ is proved in \cite{BudneyCohen05} (when $n=3$) and in \cite{K07} (when $n>3$ is odd).
\end{proof}

\subsection{The string bracket}\label{ss:string_bracket}
Similarly to \cite[\S6]{ChasSullivan99}, consider the principal $S^1$-bundle
\[
 \pi :ES^1\times\bbfemb{n}{1}\to ES^1\times_{S^1}\bbfemb{n}{1}.
\]
Let $p:E\to ES^1\times_{S^1}\bbfemb{n}{1}$ be the vector bundle of rank two associated with $\pi$, and $E_0$ the complement of the zero section of $E$.
The Gysin exact sequence for $p$ can be written as
\begin{multline*}
 \dots \to H_i(\bbfemb{n}{1})\xrightarrow{\mathrm{E}}H^{S^1}_i(\bbfemb{n}{1})\\
 \xrightarrow{c}H^{S^1}_{i-2}(\bbfemb{n}{1})\xrightarrow{\mathrm{M}} H_{i-1}(\bbfemb{n}{1})\to\dots
\end{multline*}
($\mathrm{E}$ is induced by $E_0\hookrightarrow E$, $c$ is given by capping the Euler class, and $\mathrm{M}$ is the connecting homomorphism).
Define the bracket $\{ -,-\}$ on $H^{S^1}_*(\bbfemb{n}{1})$ by
\[
 \{ x,y\} :=(-1)^{\tilde{x}}\mathrm{E}(\mathrm{M}(x)\mathrm{M}(y)).
\]
As a corollary of Theorem~\ref{thm:BV} we obtain the following.

\begin{cor}
$\{-,-\}$ is a degree two Lie bracket on $H_*^{S^1}(\bbfemb{n}{1})$.
\end{cor}

\begin{proof}
The proof is the same as that of \cite[Theorem~6.1]{ChasSullivan99}, which formally uses the BV-structure on $H_*(\bbfemb{n}{1})$, the fact that the map $\Delta$ induced by the $S^1$-action can be described as $\Delta=M\circ\mathrm{E}$, and the exactness of the Gysin sequence.
\end{proof}

\section{BV-structure on the Hochschild homology}\label{s:BV_Hochschild}
This section is independent of the previous one.
In this section we define a BV-algebra structure on the {\em Hochschild homology} $HH_*(\calO )$ associated with a {\em cyclic multiplicative operad $\calO$} in the category of {\em graded} modules.

One motivation is as follows.
When $n\ge 4$, the space $\EC{n-1}{1}$ is weakly equivalent to the homotopy totalization of an operad $fK_n$, called the \emph{framed Kontsevich operad}, which is weakly equivalent to the framed little $n$-disks operad $f\calC_n$ \cite{Salvatore06}.
There is a spectral sequence \cite{Bousfield87} converging to the homology of the homotopy totalization of a topological multiplicative operad ($fK_n$ is one of such operads), and its $E^2$-term is the Hochschild homology associated with the homology of the operad.
In general, for any multiplicative operad $\calO$ of modules, its Hochschild homology $HH_*(\calO )$ admits a Poisson algebra structure \cite{Tourtchine05}, and if moreover $\calO$ is a cyclic operad over {\em non-graded} modules, then $HH_*(\calO )$ admits a BV-algebra structure \cite{Tradler02, Menichi04}.
For $\EC{n-1}{1}$, the operad $\calO=H_*(f\calC_n)$ is a multiplicative operad of graded modules, and the Poisson structure on $HH_*(\calO)$ is proved in Salvatore's draft to coincide with that described in \cite{Budney03}.
Moreover, $f\calC_n$ is equivalent to a cyclic operad (of ``conformal $n$-balls'') \cite{Budney08}, and it turns out that $H_*(f\calC_n)$ is a cyclic multiplicative operad of {\em graded} modules.
So it is natural to ask whether $HH_*(\calO )$ admits a suitable BV-algebra structure when $\calO$ is a cyclic operad of {\em graded} modules, in such a way that it coincides with that discussed in \S\ref{s:BV} in the case of embedding spaces.
Our construction is a direct analogue to the non-graded cases.

As for operads, we follow the convention of \cite{McClureSmith04}.

\subsection{Hochschild homology}\label{ss:Hochschild}
For an operad $\calO$ and $x\in\calO (l)$, $y\in\calO (m)$, define
\[
 x\circ_i y:=x(\id ,\dots ,\id ,y,\id ,\dots ,\id )\in\calO (l+m-1),
\]
where $y$ sits in the $i$-th place, and $\id\in\calO (1)$ is the identity element.
When $\calO$ is an operad of graded modules, we denote by $\tilde{x}$ the grading of $x$ in the graded module $\calO (l)$, that is, $x\in\calO (l)_{\tilde{x}}$.

Let $\calO$ be a multiplicative operad \cite[Definition 10.1]{McClureSmith04} of graded modules;
namely $\calO$ is a non-symmetric operad of graded modules endowed with a morphism $\mathcal{ASSOC}\to\calO$, where $\mathcal{ASSOC}$ is the \emph{associative operad} given by $\mathcal{ASSOC}(n)=\{*\}$ for all $n\ge 0$.
We denote the image of $\mathcal{ASSOC}(2)=\{*\}\to\calO(2)$ by $\mu\in\calO (2)$ and call it the \emph{multiplication}.
The collection $\calO =\{\calO (k)\}_{k\ge 0}$ admits a cosimplicial module structure; the cosimplicial structure maps
\[
 d^i:\calO(k-1)\to\calO (k),\quad s^i:\calO (k+1)\to\calO (k)\quad (0\le i\le k)
\]
are defined as in \cite[\S10]{McClureSmith04} by using $\mu$ and the unit element $e\in\calO (0)$.
The grading-preserving map
\[
 \partial_k : \calO (k)\to\calO (k+1),\quad \partial_k:=d^0-d^1+\dots +(-1)^{k+1}d^{k+1}
\]
satisfies $\partial_{k+1}\partial_k=0$.
Thus we obtain a cochain complex $\{ \calO ,\partial\}$ with {\em total degree}
\[
 \abs{x}:=\tilde{x}-l
 \quad\text{for}\quad
 x\in\calO (l)
\]
(this agrees with the homological degree in the spectral sequence).
We call this the {\em Hochschild complex associated with $\calO$}.

Define the {\em normalized Hochschild complex} $\tilde{\calO}$ by
\[
 \tilde{\calO}(k):=\bigcap_{0\le i\le k-1}\ker \{s^i :\calO (k)\to\calO (k-1)\} .
\]
The following is a well-known fact.

\begin{lem}[see {\cite[III, Theorem~2.1]{GoerssJardine}} for the simplicial version]\label{lem:quasi-iso}
The map $\partial_k$ restricts to $\partial_k:\tilde{\calO}(k)\to\tilde{\calO}(k+1)$.
The inclusion map $\tilde{\calO}\to\calO$ is a quasi-isomorphism.
\end{lem}

A Poisson algebra structure on the Hochschild homology $HH(\calO ):=H_*(\calO ,\partial )$ was defined in \cite{Tourtchine05};
for $x\in\calO (l)$ and $y\in\calO (m)$, define two operations
\begin{alignat*}{2}
 &x\bullet y:=(-1)^{l\tilde{y}}\mu (x,y)&\quad&\in\calO (l+m)_{\tilde{x}+\tilde{y}},\\
 &[x,y]:=x\bar{\circ}y-(-1)^{(\abs{x}+1)(\abs{y}+1)}y\bar{\circ}x&\quad&\in\calO (l+m-1)_{\tilde{x}+\tilde{y}},
\end{alignat*}
where $\bar{\circ}$ is defined by
\[
 x\bar{\circ}y:=\sum_{1\le i\le l}(-1)^{(m-1)(l-i)+(l-1)\tilde{y}}x\circ_iy,
\]
which should be compared with the star-operation $*$ (\S\ref{ss:Poisson_emb}).

\begin{thm}[\cite{Tourtchine05}]
If $\calO$ is a multiplicative operad of graded modules, then $HH(\calO )$ is a Poisson algebra with respect to $\bullet $, $[\cdot ,\cdot ]$ and the degree $\abs{\cdot}$.
\end{thm}

\subsection{Connes' boundary operation}\label{ss:Connes}
Suppose in addition that $\calO$ is a {\em cyclic multiplicative operad} (see \cite[Definition~3.11]{Menichi04});
namely, $\calO$ is a multiplicative operad with grading-preserving linear maps
\[
 \tau_k :\calO (k)\to\calO (k)
\]
satisfying $\tau_k^{k+1}=\id$, $\tau_0(e)=e$, $\tau_2(\mu )=\mu$ and, for $x\in\calO (l)$ and $y\in\calO (m)$,
\[
 \tau_{l+m-1}(x\circ_i y)=
 \begin{cases}
  \tau_l(x)\circ_{i-1}y              & i\ge 2,\\
  (-1)^{\tilde{x}\tilde{y}}\tau_m(y)\circ_m\tau_l(x) & i=1.
 \end{cases}
\]

\begin{lem}[{\cite[Theorem 1.4 (a)]{Menichi04}}]\label{lem:cocyclic}
Let $\calO$ be a cyclic multiplicative operad of graded modules.
The collection $\{\tau_k\}_{k\ge 0}$ of maps makes the cosimplicial module $\calO$ into a cocyclic module; that is,
for $1\le i\le k$, we have
\[
 \tau_k d^i=d^{i-1}\tau_{k-1},\quad
 \tau_k s^i=s^{i-1}\tau_{k+1}.
\]
\end{lem}

Define the operation $B_k:\calO (k)\to\calO (k-1)$ by
\begin{equation}\label{eq:Connes}
 B_k(x):=(-1)^{\tilde{x}}\sum_{1\le i\le k}(-1)^{i(k-1)}\tau_{k-1}^{-i}s^{k-1}\tau_k (1-\tau_k)(x).
\end{equation}
This map is called {\em Connes' boundary operation} (for non-graded simplicial version, see \cite[(2.1.7.1)]{Loday92}).
Indeed $B$ is a boundary map:

\begin{lem}[{\cite[\S2]{Loday92}}]\label{lem:indeed_boundary}
We have $B_kB_{k+1}=0$ and $B_{k+1}\partial_k=-\partial_{k-1}B_k$.
\end{lem}

Note that $\tau_k$ does not descend to a map on $\tilde{\calO}(k)$.
But the following holds.

\begin{lem}[{\cite[\S2]{Loday92}}]\label{lem:B_restrict}
$B_k$ restricts to a map $B_k:\tilde{\calO}(k)\to\tilde{\calO}(k-1)$ of the form
\begin{equation}\label{eq:Connes2}
 B_k(x)=(-1)^{\tilde{x}}\sum_{1\le i\le k}(-1)^{i(k-1)}\tau_{k-1}^{-i}\sigma_k(x),
\end{equation}
where $\sigma_k:=s^{k-1}\tau_k$.
\end{lem}

The formula \eqref{eq:Connes} is equal to \eqref{eq:Connes2} on $\tilde{\calO}(k)$ because $s^{k-1}\tau_k^2=\tau_{k-1}s^0$ as a consequence of Lemma~\ref{lem:cocyclic}.

We have the induced map $B_k$ on Hochschild homology by Lemma~\ref{lem:indeed_boundary}.
The main result of this section is the following.

\begin{thm}\label{thm:BV_Hochschild}
$(HH(\calO ),\bullet ,[\cdot ,\cdot ], B)$ is a BV-algebra with respect to the grading $\abs{\cdot}$.
\end{thm}

This theorem has been already proved for cyclic multiplicative operads of non-graded modules \cite{Tradler02}, \cite[\S6]{Menichi04}.
The proof below is exactly same as that in \cite[\S6]{Menichi04} when the degrees $a$ and $b$ are both even.

\begin{proof}
Let $x\in\tilde{\calO}(l)_a$, $y\in\tilde{\calO}(m)_b$.
Define $Z(x,y)\in\tilde{\calO}(l+m-1)_{a+b}$ by
\[
 Z(x,y):=(-1)^{\abs{x}\abs{y}+a+b}\sum_{1\le j\le l}(-1)^{j(l+m-1)}\tau_{l+m-1}^{-j}\sigma_{l+m}(y\bullet x)
\]
and define $H(x,y)\in\tilde{\calO}(l+m-2)_{a+b}$ by $H(x,y):=\sum_{1\le j\le p\le l-1}H_{j,p}(x,y)$, where
\[
 H_{j,p}(x,y):=(-1)^{j(l-1)+(m-1)(p+1+l)+lb}\tau_{l+m-2}^{-j}\sigma_{l+m-1}(x\circ_{p-j+1}y).
\]
It is not difficult to see that the result follows from the three formulas
\begin{equation}\label{eq:formula1}
 B_{l+m}(x\bullet y)=Z(x,y)+(-1)^{\abs{x}\abs{y}}Z(y,x),
\end{equation}
\begin{multline}\label{eq:formula2}
 (-1)^{\abs{x}}\bigl( Z(x,y)-B_l(x)\bullet y\bigr) -x\bar{\circ}y\\
 =(-1)^b\partial H(x,y)+H(\partial x,y)+(-1)^{l+b+1}H(x,\partial y),
\end{multline}
\begin{equation}\label{eq:formula3}
 z\bullet w-(-1)^{\abs{z}\abs{w}}w\bullet z=(-1)^{\abs{z}}\bigl( \partial (z\bar{\circ}w)-(\partial z)\bar{\circ}w-(-1)^{\abs{z}-1}z\bar{\circ}(\partial w)\bigr) .
\end{equation}
Indeed, \eqref{eq:formula1}, \eqref{eq:formula2} and \eqref{eq:formula3} imply that
\begin{align*}
 &B_{l+m}(x\bullet y)-\bigl( B_l(x)\bullet y+(-1)^{\abs{x}}x\bullet B_m(y)+(-1)^{\abs{x}}[x,y]\bigr) \\
 &\quad =(-1)^{\abs{x}+b}\bigl( \partial H(x,y)+(-1)^b H(\partial x,y)+(-1)^{l+1}H(x,\partial y)\bigr) \\
 &\quad\quad +(-1)^{\abs{x}\abs{y}+\abs{y}+a}\bigl( \partial H(y,x)+(-1)^aH(\partial y,x)+(-1)^{m+1}H(y,\partial x)\bigr) \\
 &\quad\quad -(-1)^{(\abs{x}+1)\abs{y}}\bigl( \partial (B_m(y)\bar{\circ}x)-(\partial B_m(y))\bar{\circ}x-(-1)^{\abs{y}}B_m(y)\bar{\circ}(\partial x)\bigr) .
\end{align*}
The formula \eqref{eq:formula1} follows directly from the definition, and \eqref{eq:formula3} is \cite[(3.7)]{Tourtchine05}.
\eqref{eq:formula2} follows from the following formulas, which are proved similarly as in \cite[\S 6]{Menichi04}:
\begin{itemize}
\item $H(d^0x+(-1)^{l+1}d^{l+1}x,y)=(-1)^{\abs{x}}Z(x,y)-x\bar{\circ}y$,
\item $\displaystyle\sum_{1\le j<p\le l}H_{j,p}((-1)^{p-j}d^{p-j}x,y)=(-1)^{l+b}H(x,d^0y)$,
\item $\displaystyle\sum_{1\le j\le p\le l-1}H_{j,p}((-1)^{p-j+1}d^{p-j+1}x,y)=(-1)^{l+b}H(x,(-1)^{m+1}d^{m+1}y),$
\item $\displaystyle\sum_{1\le j\le l}H_{j,l}((-1)^{l-j+1}d^{l-j+1}x,y)=(-1)^{\abs{x}+1}B(x)\bullet y$,
\item $\displaystyle\sum_{1\le j\le p\le l}\sum_{\genfrac{}{}{0pt}{1}{1\le i\le l}{i\ne p-j,\, p-j+1}}H_{j,p}((-1)^id^ix,y)$

$\displaystyle =(-1)^{b+1}\Bigl(\sum_{1\le j\le p\le l-1}\sum_{\genfrac{}{}{0pt}{1}{1\le i\le p-1,\text{ or}}{p+m\le i\le l+m-2}}(-1)^id^iH(x,y)\Bigr. \Bigl.+d^0H(x,y)$

\qquad $+(-1)^{l+m-1}d^{l+m-1}H(x,y)\Bigr)$,
\item $\displaystyle\sum_{\genfrac{}{}{0pt}{1}{1\le j\le p\le l-1}{p\le i\le p+m-1}}(-1)^id^iH_{j,p}(x,y)=(-1)^l\sum_{1\le i\le m}(-1)^iH(x,d^iy)$.\qedhere
\end{itemize}
\end{proof}

\begin{cor}
$B_*$ defines a BV-algebra structure on $E^2$-term of the Bousfield homology spectral sequence (which converges to $H_*(\EC{n-1}{1})$ when $n\ge 4$) and descends to a BV-operation on $E^{\infty}$-term.
\end{cor}

\begin{proof}
A cyclic structure on the operad $\calC\calB_n$ of ``conformal $n$-balls'' was described in \cite{Budney08}.
An easy observation shows that $\tau_*\mu =\mu$ for the operad $H_*(\calC\calB_n)\cong H_*(f\calC_n)$, where the multiplication $\mu\in H_0(f\calC_n(2))\cong\Z$ corresponds to $1\in\Z$.
Thus $H_*(f\calC_n)$ is a cyclic multiplicative operad of graded modules, and hence $E^2_{**}=HH_*(H_*(f\calC_n))$ admits a BV-algebra structure.

The Bousfield spectral sequence \cite{Bousfield87} is derived from the double complex $\{ C_*(fK_n(*)), d, \partial_*\}$, where $fK_n$ is the framed Kontsevich operad \cite{Salvatore06} (which is cyclic and multiplicative), $C_*$ is the singular chain complex functor and $d$ is the boundary operator for singular chains.
This spectral sequence is a spectral sequence of Poisson algebras \cite{Salvatore06, K08}.
The map $B_*$ is defined on $C_*(fK_n(*))$ by \eqref{eq:Connes} and commutes with both $\partial$ and $d$ since $\tau_{k-1}$ and $s^{k-1}$ are induced by continuous maps defined on $f\calC(*)$; $\tau$ is induced by the cyclic permutation of balls, and $s$ is the forgetting map.
Thus $B_*$ commutes with all the differentials $d^r$ on $E^r$, $r\ge 2$.
\end{proof}

\begin{conj}
At least over rationals, $B$ descends to a map on $H_*(\EC{n-1}{1})$ and coincides with $\Delta$ discussed in \S\ref{s:BV}.
\end{conj}

\section*{Acknowledgment}
The author is deeply grateful to Ryan Budney and Paolo Salvatore for informing him about their framed little disks action, to Victor Turchin for answering questions about his Poisson structure on Hochschild homology, and to Takahito Naito for invaluable discussions.
The author is partially supported by Grant-in-Aid 228006, 23840015 and 25800038, JSPS.

\providecommand{\bysame}{\leavevmode\hbox to3em{\hrulefill}\thinspace}
\providecommand{\MR}{\relax\ifhmode\unskip\space\fi MR }
\providecommand{\MRhref}[2]{%
  \href{http://www.ams.org/mathscinet-getitem?mr=#1}{#2}
}
\providecommand{\href}[2]{#2}

\end{document}